\documentclass{amsart}
\usepackage{geometry}
\usepackage{amsmath,amsfonts,amssymb}
\usepackage{amsthm}
\usepackage{ifthen}
\usepackage{longtable}
\usepackage{stmaryrd}
\usepackage{array}
\usepackage{url}
\usepackage{verbatim}
\usepackage{booktabs}
\usepackage[table]{xcolor}
\usepackage{longtable}

\usepackage{color}
\definecolor{darkgreen}{rgb}{0,0.5,0}
\usepackage[colorlinks, citecolor=darkgreen, backref]{hyperref}

\newcommand{\Z}{\mathbb{Z}}
\newcommand{\Q}{\mathbb{Q}}
\newcommand{\R}{\mathbb{R}}
\newcommand{\C}{\mathbb{C}}

\newcommand{\ord}{\mathcal{O}}

\newcommand{\abs}[1]{\left|#1\right|}

\newcommand{\LL}{\mathcal{L}}

\newcommand{\NK}{N_K}
\newcommand{\eps}{\epsilon}
\newcommand{\CLone}{A_L}
\newcommand{\CLtwo}{B_L}
\newcommand{\CLtwofour}{C_L}

\newcommand{\numquinticsols}{255}

\DeclareMathOperator{\Gal}{Gal} 
\DeclareMathOperator{\Reg}{Reg}

\DeclareMathOperator{\Trace}{Tr}

\DeclareMathOperator{\Hom}{Hom}

\newtheorem*{theorem*}{Theorem}
\newtheorem{theorem}{Theorem}
\newtheorem{lemma}{Lemma}
\newtheorem{corollary}{Corollary}
\newtheorem{proposition}{Proposition}
\newtheorem{conjecture}{Conjecture}

\theoremstyle{remark}

\begin{document}
\title[Sum of two units]{On integers that are representable as the sum of two units}
\thanks{R. V. was supported by {Charles University} programme PRIMUS/24/SCI/010 and {Charles University} programme UNCE/24/SCI/022.}

\author[R. Visser]{Robin Visser}
\address{R. Visser, Charles University\\ Faculty of Mathematics and Physics\\ Department of Algebra \\ Sokolovsk\'{a} 83\\ 186 75 Praha 8\\ Czech Republic}
\email{robin.visser@matfyz.cuni.cz}

\author[V. Ziegler]{Volker Ziegler}
\address{V. Ziegler,
University of Salzburg,
Hellbrunnerstrasse 34,
A-5020 Salzburg, Austria}
\email{volker.ziegler\char'100plus.ac.at}

\begin{abstract}
Let $K$ be a number field of degree $D$ with maximal order $\ord_K$. We show that under certain conditions on $K$, which in particular are always satisfied if $D$ is odd or if $D\geq 3$ and $K$ is primitive, the set of positive integers $\NK$ that can be expressed as a sum of two units in $\ord_K^*$ is a finite effectively computable set.  This result partially resolves an open problem posed by Tinkov\'{a}, Yatsyna, and the first author.
We illustrate our method by explicitly computing $\NK$ for the smallest totally real quintic field $K$ with Galois group $S_5$.
\end{abstract}

\maketitle

\section{Introduction}

Let $K$ be a number field and $\ord_K$ be its maximal order. We denote by $\NK$ the set of all positive integers $n$ such that $n=\eps+\delta$ with units $\eps,\delta \in \ord_K^*$. Recently, it has been shown by Tinkov\'{a}, Yatsyna, and the first author \cite{Tinkova:2025} that $N_K$ is infinite if and only if $K$ contains a real quadratic subfield. In fact, let $L\subseteq K$ be a real quadratic subfield and let $\eps\in\ord_L^*$ be a unit in $L$. Then all integers of the form
\begin{equation} \label{eq:realquadsols}
    n=\eps^k+\bar\eps^k=\Trace_{L/\Q}(\eps^k),
\end{equation}
where $\bar \eps$ is the (Galois) conjugate of $\eps$ and $k$ is some integer, are contained in $N_K$. 

Let $N_K^{\text{quad}}$ be the set of all positive integers $n$ such that there exists an $\eps\in \ord_K^*$ with $\Q(\eps)$ a real quadratic field such that $n=\Trace_{\Q(\eps)/\Q}(\eps)$. Then we write
$\NK^- = \NK\setminus \NK^{\text{quad}}$. Although the methods in \cite{Tinkova:2025} show that $N_K^-$ is finite for all number fields $K$, their proof relies on the finiteness of solutions to unit equations in several unknowns and thus cannot effectively determine the set $N_K^-$ for an arbitrary number field $K$.

The main result in this paper is to provide an effective upper bound for the largest member of $N_K^-$. Let us denote by $h(\cdot)$ the absolute logarithmic Weil height (for a definition see Section \ref{sec:preliminaries}), then we prove the following:

\begin{theorem}\label{th:gen-bound}
 Let $K$ be a given number field. 
 Assume that for any nonzero rational integer $n$,
 there does not exist a unit $\eps \in \ord_K^* \backslash \Q$ such that $n - \eps$ is a Galois conjugate of $\eps$.
 If $n\neq 0$ is a rational integer and $\eps,\delta \in \ord_K^*$ are units such that $n=\eps+\delta$, then
 $$\max\{\log |n|,h(\eps),h(\delta)\}\leq C$$
 where $C$ is an effectively computable constant depending only on $K$.
\end{theorem}

To give an explicit example of such a constant $C$, we prove that we can take the constant $C$ to be:
\begin{equation*}
    C := \max_{L \subseteq K} \Big\{ 2 \CLone^2 \CLtwo \max \big\{  e \cdot r_L t_L ( \CLone \CLtwo + 1), \; 4 \CLtwo \log(4r_L) + 4 \CLtwofour \log ( 2 \CLtwofour ) \big\} \Big\} ,
\end{equation*}
where the maximum is taken over all subfields $L$ of $K$, and where for each subfield $L \subseteq K$, we define the real constants $\CLone, \CLtwo, \CLtwofour$ as:
\begin{align*}
    \CLone &:= t_L \cdot \max_{i = 1, \dots, t_L} \left\{ \log \abs{ \eta_i^{(1)} } \right\}, \quad
    \CLtwo := \frac{t_L}{\Reg_L} \cdot \max_{1 \leq i, j \leq t_L} \{ \abs{ \det R_{L,i,j} } \}, \quad 
    \text{and} \\
    \CLtwofour &:= \CLtwo \cdot 2^{12t_L + 27} D^2 (D-1)^2 \big( 1 + \log (D(D-1)) \big) \cdot \pi \left(  \prod_{k=1}^{t_L} h_m(\eta_k)  \right)^2 ,
\end{align*}
where $D$ is the degree of $K$, $t_L$ is the unit rank of $\ord_L^*$, $r_L$ is the order of the torsion subgroup of $\ord_L^*$,   $\eta_1, \eta_2, \dots, \eta_{t_L} \in \ord_L^*$ are a fundamental system of units for $L$, $\Reg_L$ is the regulator of $L$, $R_{L,i,j}$ is the matrix formed by removing the $i$-th column and $j$-th row from the matrix $R_L := (\log | \eta_i^{(j)} | )_{1\leq i,j\leq t_L}$, and $h_m(\cdot)$ denotes the modified height, as defined in Section~\ref{sec:preliminaries}.

As a consequence, we can obtain constraints on the Galois groups of such number fields $K$ satisfying the conditions of Theorem~\ref{th:gen-bound}, and can show the following corollary:

\begin{corollary} \label{cor:galgroup}
    Let $K$ be a number field.
    Assume that for every subfield $L \subseteq K$,
    either its absolute degree $d = [L : \Q]$ is odd,
    or the Galois group of its normal closure, viewed as a transitive permutation group of degree $d$, 
    is not contained in the wreath product $C_2 \wr S_{d/2}$. 
    Then, for any nonzero rational integer $n$ and units $\eps, \delta \in \ord_K^*$ satisfying $\eps + \delta = n$, we have
    \begin{equation*}
        \max\{\log |n|,h(\eps),h(\delta)\}\leq C,
    \end{equation*}
    where $C$ is an effectively computable constant depending only on $K$, as given in Theorem~\ref{th:gen-bound}.
    In particular, there exists an effectively computable algorithm to determine $\NK$ if $D = [K : \Q]$ is odd, or if $D \geq 3$ and $K$ is primitive.
\end{corollary}

Unfortunately we are unable to resolve the general number field case.  However, apart from the infinite family of conjugate real quadratic units given in \eqref{eq:realquadsols}, we conjecture that all other solutions to the unit equation $\eps + \delta = n$ can be effectively bounded:

\begin{conjecture}
 Let $K$ be a number field. Then there exists an effectively computable constant $C'$ depending only on $K$, with the following property.
 Let $n$ be a nonzero rational integer and let $\eps, \delta \in \mathcal{O}_K^*$ be units satisfying $\eps + \delta = n$. Assume that $\eps$ and $\delta$ do not arise as in \eqref{eq:realquadsols}, that is, assume $\Q(\eps)$ is not a real quadratic field with $\delta$ the Galois conjugate of $\eps$. Then 
\begin{equation*}
    \max\{\log |n|,h(\eps),h(\delta)\}\leq C' .
\end{equation*}
In particular, $\NK^-$ is an effectively computable finite set.
\end{conjecture}

We demonstrate the strength of our method by computing $\NK$ for a given degree 5 example:

\begin{theorem}\label{th:quintic}
 Let $K=\Q(\alpha)$, where $\alpha$ is a root of $X^5-5X^3-X^2+5X+1$. Then there are exactly $\numquinticsols$ non-equivalent pairs of units $(\eps, \delta)$ such that $\eps + \delta$ is a non-zero rational integer. In particular, we have
 $$\NK=\{1,2,3,4,5,6,7,10,11,15,251\}.$$
\end{theorem}

Here, two solutions $(\eps_1, \delta_1, n_1)$ and $(\eps_2, \delta_2, n_2)$ are considered equivalent if either $(\eps_1, \delta_1, n_1) = (\pm \sigma(\eps_2), \pm \sigma(\delta_2), \pm n_2)$ or $(\eps_1, \delta_1, n_1) = (\pm \sigma(\delta_2), \pm \sigma(\eps_2), \pm n_2)$ for some $\sigma \in \Gal(K/\Q)$.
A full list of all solutions is given at the end of paper in Table~\ref{tab:quinticsolutions}. Note that the field $K$ in Theorem~\ref{th:quintic} has LMFDB label 5.5.24217.1 \cite{lmfdb_quintic} and is the smallest totally real quintic field, whose Galois closure has Galois group isomorphic to $S_5$ (see the number field tables due to Diaz y Diaz \cite{Diaz:1991}).

\bigskip

We briefly recall some known results about the sets $\NK$. In the case of prime cyclotomic fields $K = \Q(\zeta_p)$, Newman \cite[p.~89]{Newman:1974} posed the question of explicitly determining $\NK$, observing that $\{1, 2, 3\} \subseteq N_{\Q(\zeta_p)}$ for all primes $p > 3$.  Kostra \cite{Kostra:1994} and Newman \cite{Newman:1993} independently showed that $kp \not \in N_{\Q(\zeta_p)}$ for all primes $p \geq 3$ and all integers $k \in \Z$.

Several papers have studied the related problem of classifying the number fields $K$ such that $1 \in \NK$.  In a series of papers, culminating in \cite{Nagell:1969}, Nagell classified all fields $K$ of unit rank 0 or 1 with the property that $1 \in \NK$. Triantafillou \cite{Triantafillou:2021} showed that if 3 splits completely in a number field $K$ and $3 \nmid [K : \Q]$, then $1 \not \in \NK$.  Recently, Freitas--Kraus--Siksek \cite{FKS:2021} have shown that, for any given prime $p \not = 3$, there are only finitely many cyclic degree $p$ fields $K$ such that $1 \in \NK$.  We refer the reader to \cite[Section~5.5]{EvertseGyory2015} and the references therein for further background.

There has been some considerable recent progress on the equation $\eps + \delta = n$ particularly over cubic number fields. Vukusic and the second author \cite{VZ} computed the solutions to $\eps + \delta = n$ over the family of Shanks' simplest cubic fields $K_a = \Q(\rho_a)$ where $\rho_a$ is a root of the cubic polynomial $f(x) = x^3 - ax^2 - (a+3)x - 1$.  In particular, they classified all solutions where $\eps, \delta \in \Z[\rho_a]^*$ under the assumption that $n \leq \max (|a|^{1/3}, 1)$.  More recently, Komatsu \cite{Komatsu} as well as Tinkov\'{a}, Yatsyna, and the first author \cite{Tinkova:2025} computed $\NK$ for all such cyclic cubic fields $K$, proving that if $n = \eps + \delta$ for some units $\eps, \delta \in \ord_K^*$ in a cubic cyclic field $K$, then $n \in \{1, 2, 3, 4, 5, 7, 19, 22 \}$.  A similar theorem for complex cubic fields was also given in \cite[Theorem~1.3]{Tinkova:2025}. We also mention recent work of Khawaja and Siksek \cite{KhawajaSiksek2026} which studies the unit equation over cubic fields from a statistical point of view, showing that for any fixed integer $n \not \in \{-2, 0, 2\}$, the equation $\eps + \delta = n$ has no solutions for $\eps, \delta \in \mathcal{O}_K^*$ for 100\% of cubic fields $K$, when ordered by discriminant.

\bigskip
The paper is structured as follows. In Section~\ref{sec:preliminaries}, we begin by stating some various preliminary results regarding linear forms in logarithms and LLL basis reduction.  In Section~\ref{sec:unit_matching}, we apply the recent unit matching techniques of Bajpai and Bennett \cite{Bajpai:2023} to bound the heights of $\eps, \delta$ and give a proof of Theorem~\ref{th:gen-bound}. In Section~\ref{sec:galoisgroups}, we prove some results which constrain the possible Galois groups of fields $K$ not satisfying the conditions of Theorem~\ref{th:gen-bound}, thus proving Corollary~\ref{cor:galgroup}.  Finally, in Section~\ref{sec:quintic_example}, we apply Theorem~\ref{th:gen-bound} along with LLL-reduction to the smallest degree 5 field $K = \Q(\alpha)$ with Galois group $S_5$, thereby giving a proof of Theorem~\ref{th:quintic}.

\bigskip
We are grateful to Magdal\'{e}na Tinkov\'{a} and Pavlo Yatsyna for many useful comments and suggestions on an earlier draft of this paper.

\section{Preliminaries}\label{sec:preliminaries}

In this section we collect some helpful results for the proof of our main theorems. Let us start with the notion of height. Assume that $\alpha \in \overline{\Q}$ is an algebraic number with minimal polynomial
$$a \prod_{i=1}^d (X-\alpha_i) \in \Z[X].$$
Then we call
$$
h(\alpha):=\frac 1d \left(\log |a| +\sum_{i=1}^d \max\{0,\log |\alpha_i|\}\right)
$$
the \emph{absolute logarithmic Weil height} of $\alpha$. 
We furthermore call
\begin{equation*}
    h_m(\alpha) := 
    \max \left\{  d h(\alpha), \, | \log \alpha_i |, \, 0.16  \right\}
\end{equation*}
the \emph{modified height} of $\alpha$ .

\subsection{Linear forms in logarithms}
To find the effective bounds in our main theorems we will apply lower bounds for linear forms in logarithms. In particular, for concrete computations we will apply the effective bounds found by Matveev \cite[Corollary~2.3]{Matveev:2000}:

\begin{lemma}\label{lem:matveev}
  Denote by $\alpha_1, \dots, \alpha_N$ algebraic numbers, not $0$ nor $1$. Let $K=\Q(\alpha_1,\ldots,\alpha_N)$ and $D=[K:\Q]$ and denote by $b_1, \dots, b_N$ rational integers with $b_N\neq 0$. Furthermore, let $\kappa=1$ if $K$ is real and $\kappa=2$ otherwise. For all integers $j$ with $1\leq j\leq N$ choose
 $$
    A_j\geq \max\{D h(\alpha_j), |\log\alpha_j|,0.16\},
  $$
  and set
  $$
    E=\max ( \{1\} \cup \{|b_j| A_j /A_N\: :\: 1\leq j \leq N \} ).
  $$
  Assume that $\Lambda = b_1 \log \alpha_1 + \cdots + b_N \log \alpha_N\neq 0$. Then
  \begin{equation*}
    \log |\Lambda|
    > -C(N,\kappa)D^2 \Omega\log(eD)\log(eE)
  \end{equation*}
  with $\Omega=A_1\cdots A_N$ and
  $$C(N,\kappa)= \min\left\{\frac 1{\kappa}\left(\frac {eN}2\right)^\kappa 30^{N+3} N^{3.5}, 2^{6N+20}\right\}.$$ 
  \end{lemma}

The next lemma, due to Peth\H{o} and de Weger \cite{Pethoe:1986} is also helpful to compute the upper bound stated in Theorem \ref{th:gen-bound}.

\begin{lemma}\label{lem:PdW}
 Let $u,v \geq 0, h \geq 1$ and $x \in \R$ be the largest solution of $x=u+v(\log{x})^h$. Then
$$
x<\max\{2^h(u^{1/h}+v^{1/h}\log(h^hv))^h, 2^h(u^{1/h}+2e^2)^h\}.
$$
\end{lemma}

For a proof of Lemma \ref{lem:PdW} we refer to \cite[Appendix B]{Smart:DiGl}.

\subsection{LLL basis reduction}
Let $\LL\subseteq \R^k$ be a $k$-dimensional lattice with LLL-reduced basis
$b_1,\dots,b_k$ and denote by $B$ the matrix with columns $b_1,\dots, b_k$. 
Moreover, we denote by $b^*_1,\dots,b^*_k$ the orthogonal basis of $\R^k$ which we obtain by applying the Gram-Schmidt process to the basis $b_1,\dots,b_k$. In particular, we have that
$$b^*_i=b_i-\sum_{j=1}^{i-1}\mu_{i,j}b^*_j, \qquad \mu_{i,j}=\frac{\langle b_i,b_j\rangle}{\langle b_j^*,b_j^*\rangle}.$$
Further, let us define
$$l(\LL)=\min_{0 \neq x \in \LL} \|x\|,$$
where $\|\cdot\|$ denotes the euclidean norm on $\R^k$. It is well known,
that by applying the LLL-algorithm it is possible to give in polynomial time
a lower bound $\tilde c_1$ for $l(\LL)$ (see e.g. \cite[Section 5.4]{Smart:DiGl}).

\begin{lemma}\label{lem:lattice}
 Let 
 $$
 \tilde c_2=\max_{1\leq j\leq k}\left\{\frac{\|b_1\|^2}{\|b_j^*\|^2} \right\}.
 $$
 Then we have
 $$l(\LL)^2\geq \tilde c_2^{-1} \|b_1\|^2=\tilde c_1.$$
\end{lemma}

In our proof of Theorem \ref{th:quintic} we are interested in finding good lower bounds for a linear form
\begin{equation} \label{eq:redform1}
|x_1\eta_1+\dots+x_k\eta_k|,
\end{equation} 
where $\eta_1,\eta_2,\dots,\eta_k$ are given real numbers that are linearly independent over $\Q$. Moreover, we suppose that the integers $x_i$ are bounded, that is we assume that $|x_i| \leq X_i$, where $X_i$, with $1\leq i\leq k$, are given.

The basic idea in such a situation, due to de Weger \cite{deWeger:1987} (see also \cite[Section VI.3]{Smart:DiGl}), is to approximate the linear form
\eqref{eq:redform1} by an approximation lattice. Namely, we consider the lattice $\LL$ generated by the columns of the matrix
\begin{equation*}
\mathcal{A}=\begin{pmatrix}
    1 & 0 & \dots & 0  & 0 \\
    0 & 1 & \dots & 0  & 0 \\
    \vdots & \vdots & \ddots & \vdots & \vdots \\
    0 & 0 & \dots & 1  & 0 \\
    \left[{\tilde C\eta_1}\right] & \left[{\tilde C\eta_2}\right] & \dots & \left[{\tilde C\eta_{k-1}}\right] & \left[{\tilde C\eta_k}\right]
\end{pmatrix},
\end{equation*}
where $\tilde C$ is a large constant usually of the size of about $\max_{1\leq i\leq k}\{X_i\}^k$, and where $\big[ \tilde{C} \eta_i \big]$ denotes the nearest integer to $\tilde{C} \eta_i$ (with any fixed convention taken for half-integers). Let us assume that we have an LLL-reduced basis $b_1,\dots,b_k$ of $\LL$ and that we have a lower bound
$l(\LL)\geq \tilde c_1$. Note that $\tilde c_1$ can be computed by using the results of Lemma~\ref{lem:lattice}. Then we have with these notations the following lemma (for a similar statement see \cite[Lemma VI.1]{Smart:DiGl}):

\begin{lemma}\label{lem:real-reduce}
Let $M\subset \Z^k$ be a finite set of integer $k$-tuples. 
Assume that 
$$S=\max_{(x_1,\dots,x_k)\in M}\left\{\sum_{i=1}^{k-1}x_i^2\right\}$$
and
$$T=\max_{(x_1,\dots,x_k)\in M}\left\{\frac{\sum_{i=1}^k{|x_i|}}{2}\right\}.$$ If $\tilde c_1^2 \ge T^2+S$, then we have either
$x_1=x_2=\dots=x_{k}=0$ or
\begin{equation} \label{eq:reduction-real}
\frac{\sqrt{\tilde c_1^2-S}-T}{\tilde C}<|x_1\eta_1+\dots+x_k\eta_k|
\end{equation}
holds for all $(x_1,\dots,x_k)\in M$ .
\end{lemma}

We give a proof of Lemma \ref{lem:real-reduce}, since it is slightly more general, than usually stated. 

\begin{proof}
We closely follow the proof given in \cite[Lemma VI.1]{Smart:DiGl}. We put 
$$\Phi=\sum_{i=1}^k x_i\left[\tilde C\eta_i\right]$$
and note that
$$\left|\Phi-\tilde C\cdot\sum_{i=1}^k x_i\eta_i\right|\leq \frac 12 \sum_{i=1}^k |x_i|\leq T.$$ 
Hence, we have
$$|\Phi|\leq T+|x_1\eta_1+\dots+x_k\eta_k|.$$
Next, we consider the lattice point 
$$\mathbf x=\mathcal A \left(\begin{array}{c} x_1\\ \vdots\\ x_{k-1}\\ x_k\end{array}\right)=\left(\begin{array}{c} x_1\\ \vdots\\ x_{k-1}\\ \Phi \end{array}\right).$$
If $\mathbf x$ is not the zero vector, then we have
$$
\tilde c_1^2\leq l(\LL)^2\leq \|\mathbf x\|^2=\sum_{i=1}^{k-1} x_i^2+\Phi^2\leq S+(T+|x_1\eta_1+\dots+x_k\eta_k|)^2.
$$
Since we assume $\tilde c_1^2 \ge T^2+S$ we obtain \eqref{eq:reduction-real}.
\end{proof}

\section{Unit matching} \label{sec:unit_matching}

Let us assume that $\eps+\delta=n$ with $\eps,\delta \in \ord_K^*$. Let us consider the field $L=\Q(\eps)$. Since $\delta=n-\eps$ we have that $\Q(\delta)=\Q(\eps)=L$. Let us write $[L:\Q]=d$.

Next, we consider the set of embeddings
$$\Hom_\Q(L,\C)=\{\tau_1,\dots,\tau_d\}$$
and the lifts $\sigma_1,\dots,\sigma_d\in \Hom_\Q(K,\C)$ of embeddings $\tau_1,\dots,\tau_d$ such that $\sigma_i|_L=\tau_i$ for $i=1,\dots,d$. For any $\alpha\in K$, we will write $\alpha^{(i)}=\sigma_i(\alpha)$. Let us note that $\eps^{(1)}, \dots, \eps^{(d)}$ and $\delta^{(1)}, \dots, \delta^{(d)}$ are both a complete system of conjugates of $\eps$ and $\delta$ respectively.

Let us assume that the unit rank of $\ord_L^*$ is $t$ and that $\eta_1,\dots,\eta_t\in \ord_L^*$ is a fundamental system of units. We denote by $\zeta$ a primitive root of unity of $L$, that is, $\zeta$ generates the torsion group $\mu(L)$ of $L^*$ and let us write $r=\sharp \mu(L)$ for the order of $\zeta$. Hence, there are integers
$x_0,x_1,\dots,x_t$ and $y_0,y_1,\dots,y_t$ such that
$$\eps=\zeta^{x_0}\eta_1^{x_1}\cdots \eta_t^{x_t} \quad\text{and}\quad
\delta=\zeta^{y_0}\eta_1^{y_1}\cdots \eta_t^{y_t}.$$
Let us write 
$$
X:=\max_{i=1,\dots,t}\{|x_i|\} \qquad \text{and}\qquad Y:=\max_{i=1,\dots,t}\{|y_i|\}.
$$
In order to prove Theorem \ref{th:gen-bound} we have to find bounds for $X$ and $Y$, unless $\eps$ and $\delta$ are a pair of conjugate quadratic algebraic integers.

Without loss of generality, we may now assume the following:
$$\abs{\eps^{(1)}}\geq \abs{\eps^{(2)}}\geq \dots \geq \abs{\eps^{(d)}}$$
and
$$\abs{\eps^{(1)}}\geq \max_{i=1,\dots,d}\left\{\abs{\delta^{(i)}}\right\}.$$
We consider the matrix 
$$R=\left(\log\abs{\eta_i^{(j)}}\right)_{1\leq i,j\leq t}$$
and the matrices $R_{i,j}$ for $i=1,\dots,t$ which are obtained from $R$ by deleting the $i$-th column and $j$-th row. Then the regulator of $L$ is defined by
$$\Reg_L=\abs{\det \left(\log\abs{\eta_i^{(j)}}\right)_{1\leq i,j\leq t}}.$$
With these assumptions and notations we have.

\begin{lemma}\label{lem:X-eps-bound}
We have that 
$$C_1 \log\abs{\eps^{(1)}} \leq X \leq C_2 \log \abs{\eps^{(1)}}$$
and $Y< C_2 \log\abs{\eps^{(1)}}$, where
$$C_1=\frac 1{t\max_{i=1,\dots,t}\left\{\log\abs{\eta^{(1)}_i}\right\}}\quad \text{and}\quad C_2=t\frac{\max_{1\leq i,j\leq t}\left\{\abs{\det R_{i,j}}\right\}}{\Reg_L}.$$

Moreover, we have $Y< C_3 X$, where $C_3=C_2/C_1$.
\end{lemma}

\begin{proof}
We have
$$\log\abs{\eps^{(1)}}=x_1\log\abs{\eta^{(1)}_1}+\dots+x_t\log\abs{\eta^{(1)}_t}\leq X t\max_{i=1,\dots,t}\left\{\log\abs{\eta^{(1)}_i}\right\}$$
which proves the lower bound for $X$.

For the upper bound for $X$ we consider the linear system of equations
$$x_1\log \abs{\eta_1^{(i)}}+\dots+x_t\log \abs{\eta_t^{(i)}}=\log\abs{\eps^{(i)}} \qquad 1\leq i\leq t.$$
By an application of Cramer's rule we obtain
$$x_j=\frac{\sum_{i=1}^t \log\abs{\eps^{(i)}}\det(R_{i,j})}{\Reg_L}\leq \log\abs{\eps^{(1)}} t\frac{\max_{1\leq i,j\leq t}\left\{\abs{\det R_{i,j}}\right\}}{\Reg_L},$$
which proves the upper bound for $X$. Since $\abs{\delta^{(i)}} \leq\abs{\eps^{(1)}}$ for all $1\leq i\leq t$, we also obtain the upper bound for $Y$.

The last statement follows from the following observation:
$$Y<C_2 \log\abs{\eps^{(1)}}< C_2 \frac{X}{C_1}.$$
\end{proof}

With these notations we have for all $i=1,\dots, d$ that
$$\eps^{(i)}+\delta^{(i)}=n$$
which implies the following system of unit equations
\begin{equation}\label{eq:UEQ-Sys}
\eps^{(i)}+\delta^{(i)}-\eps^{(j)}-\delta^{(j)}=0,\qquad 1\leq i<j\leq d .
\end{equation}

\begin{lemma}\label{lem:matching}
Assume that $\eps^{(i)},\delta^{(i)},\eps^{(j)},\delta^{(j)}\in \ord_K^*$ are all pairwise distinct and satisfy
$$\eps^{(i)}+\delta^{(i)}-\eps^{(j)}-\delta^{(j)}=0$$
with $i<j$. Let 
$$m={\min}^{(2)}\left(\abs{\eps^{(i)}},\abs{\eps^{(j)}},\abs{\delta^{(i)}},\abs{\delta^{(j)}}\right)$$
be the second smallest absolute value of the involved units. Then there exist effectively computable constants $C_4$ and $C_5$, depending only on $K$, such that
$m\geq \frac{\abs{\eps^{(i)}}}{4r}\exp (-C_4\log X)$ or $X\leq C_5$.
\end{lemma}

Explicitly, one can take the constants $C_4$ and $C_5$ in all cases to be
\begin{align*}
    C_4 &= 2^{12t+27} D^2 (D-1)^2 (1 + \ln(D(D-1))) \cdot \pi \left( \prod_{k=1}^t h_m(\eta_k) \right)^2 , \quad \text{and} \\
    C_5 &= ert(C_3+1) ,
\end{align*}
where $C_3 = C_2/C_1$ is as defined in Lemma~\ref{lem:X-eps-bound}.

\begin{proof}
Note that the case $\abs{\eps^{(i)}}=m$ is trivial and we may assume that $\abs{\eps^{(i)}}\neq m$. 

    First, let us assume that
    $$\abs{\eps^{(j)}}=\max\left\{\abs{\eps^{(j)}},\abs{\delta^{(i)}},\abs{\delta^{(j)}}\right\}.$$

As $\abs{\eps^{(i)}} \geq \abs{\eps^{(j)}}$, this therefore implies $\abs{\eps^{(i)}}\geq \abs{\eps^{(j)}}\geq m = \max ( \abs{\delta^{(i)}},\abs{\delta^{(j)}} ) $. We have $\eps^{(j)}-\eps^{(i)}=\delta^{(i)}-\delta^{(j)}$ and we obtain by dividing through $\eps^{(i)}$ the following inequality
$$\abs{\frac{\eps^{(j)}}{\eps^{(i)}}-1}=\abs{\frac{\delta^{(i)}-\delta^{(j)}}{\eps^{(i)}}}\leq \frac{\abs{\delta^{(i)}}+\abs{\delta^{(j)}}}{\abs{\eps^{(i)}}}\leq \frac{2m}{\abs{\eps^{(i)}}}.$$
Note that $|\log (1+x)|<2|x|$ for all $|x|<\frac 12$ holds. In the case that
$$ \abs{\frac{\eps^{(j)}}{\eps^{(i)}}-1} >\frac 12$$
we obtain $\frac{\abs{\eps^{(i)}}}{4}<m$ which proves our lemma with $C_4=C_5=0$. Thus we may assume that $ |\eps^{(i)}/\eps^{(j)} - 1 | < \frac 12$. Hence, we obtain
$$\stackrel{:=\Lambda}{\abs{\overbrace{x_0\log \frac{\zeta^{(j)}}{\zeta^{(i)}}+x_1 \log \frac{\eta_1^{(j)}}{\eta_1^{(i)}} +\dots+x_t\log \frac{\eta_t^{(j)}}{\eta_t^{(i)}} -k\pi i}}}\leq\abs{\log\left(\frac{\eps^{(j)}}{\eps^{(i)}}\right)}\leq \frac{4m}{\abs{\eps^{(i)}}}$$
for some $k\in \Z$. Here we choose $k$ such that the imaginary part of $\Lambda$ lies within the interval $[-\pi,\pi)$. Let us note that $r\log \frac{\zeta^{(j)}}{\zeta^{(i)}}=z\pi i$ for some integer $z$. Such that we get
$$\abs{rx_1 \log \frac{\eta_1^{(j)}}{\eta_1^{(i)}} +\dots+rx_t\log \frac{\eta_t^{(j)}}{\eta_t^{(i)}} -\tilde k \pi i}\leq \frac{4mr}{\abs{\eps^{(i)}}}.$$

Therefore we apply Matveev's theorem to $r\Lambda$. That is, we have $N=t+1$,
$\alpha_\ell=\frac{\eta_\ell^{(j)}}{\eta_\ell^{(i)}}$ for $1\leq \ell \leq t$ and $\alpha_{t+1}=-1$. Moreover, we have $b_\ell=rx_\ell$ and $b_{t+1}=\tilde k$, with 
$$\abs{\tilde k}=\abs{rk-rx_0\frac{\log (\zeta^{(j)}/\zeta^{(i)})}{2\pi i}}<rtX.$$
Note that we deal with the logarithms of algebraic numbers from the compositum $K^{(i)}K^{(j)}$, which is a field of degree at most $D(D-1)$. Here $K^{(i)}$ and $K^{(j)}$ denote the corresponding conjugate fields of $K$. That is we choose $A_{t+1}=\pi$ and for $k=1,\dots,t$ we choose
$$A_k=2\cdot h_m(\eta_k)=\max\left\{2D(D-1)h(\eta_k),2\left|\log \eta_k^{(j)}\right|,0.32\right\},$$
where the index $j$ runs from $1$ to $D$. Note that since $h\left(\eta_k^{(j)}/\eta_k^{(i)}\right)\leq 2 h(\eta_k)$ and 
$$\left|\log (\eta_k^{(j)}/\eta_k^{(i)})\right|\leq 2 \max_{1\leq i \leq D} \left|\log \eta_k^{(i)}\right|$$
the choice of $A_k$ is suitable for an application of Matveev's theorem. Next, we want to note that by reordering the indices we may assume that $A_N=\max_{1\leq i \leq N} A_i$ that is we have $E\leq rtX$. Therefore an application of Matveev's theorem yields
\begin{equation*}
    \log |r\Lambda| > - 2^{7t+26} D^2 (D-1)^2\left(1+\log(D(D-1))\right)\pi\left(\prod_{k=1}^t h_m(\eta_k)\right)\log (ert X).
\end{equation*}
Note that since $\eps^{(i)}\neq \eps^{(j)}$ we have $\Lambda\neq 0$ and the application of Matveev's theorem is justified.

Assuming that $X>ert=C_5$ we have 
$$m\geq \frac{\left|\eps^{(i)}\right|}{4r}\exp(-C_4\log X),$$
with
$$C_4=2^{7t+27}D^2(D-1)^2\left(1+\log(D(D-1))\right)\pi\left(\prod_{k=1}^t h_m(\eta_k)\right).$$

Let us note that in the case that $K$ is a totally real field, we can choose $r=1$ and can omit the $k\pi i$ term in the definition of $\Lambda$. That is we have $N=t$ and we obtain in this case the slightly improved constant 
$$C_4=2^{7t+21}D^2(D-1)^2\left(1+\log(D(D-1))\right)\left(\prod_{k=1}^t h_m(\eta_k)\right).$$

Next, let us assume that
    $$\abs{\delta^{(j)}}=\max\left\{\abs{\eps^{(j)}},\abs{\delta^{(i)}},\abs{\delta^{(j)}}\right\}.$$
Then we obtain
$$\abs{\frac{\delta^{(j)}}{\eps^{(i)}}-1}\leq \frac{\abs{\delta^{(i)}}+\abs{\eps^{(j)}}}{\abs{\eps^{(i)}}}\leq \frac{2m}{\abs{\eps^{(i)}}}.$$
Similarly as in the case before we obtain
$$|r\Lambda|\geq \frac{4rm}{\abs{\eps^{(i)}}}$$
with
$$r\Lambda= ry_1 \log \eta_1^{(j)} +\dots + ry_t \log \eta_t^{(j)}-rx_1 \log \eta_1^{(i)} -\dots - rx_t \log \eta_t^{(i)} +\tilde k \pi i.$$
Since by assumption $\eps^{(i)}\neq \delta^{(j)}$ we also have $\Lambda\neq 0$. We apply Matveev's theorem with $N=2t+1$, and $\alpha_\ell=\eta_\ell^{(j)}$ and $b_\ell=ry_\ell$ for $1\leq \ell \leq t$ and $\alpha_{t+\ell}=\eta_\ell^{(i)}$ and $b_{t+\ell}=-rx_\ell$ for $1\leq \ell \leq t$. We also put $\alpha_{2t+1}=-1$ and $b_{2t+1}=\tilde k$. Note that
$\tilde k<rtX+rtY<rt(C_3+1)X$, which implies that we can choose $E=rt(C_3+1)X$. We choose $A_{2t+1}=\pi$ and for $k=1,\dots,t$ we choose
$$A_k=A_{t+k}=h_m(\eta_k):=\max\left\{D(D-1)h(\eta_k),\left|\log \eta_k^{(j)}\right|,0.16\right\},$$
where the index $j$ runs from $1$ to $D$. Therefore an application of Matveev's theorem yields  
\begin{multline*}
    \log |r\Lambda|>- 2^{12t+26}D^2(D-1)^2\left(1+\log(D(D-1))\right)\pi\\
    \times \left(\prod_{k=1}^t h_m(\eta_k)\right)^2 \log\left(ert(C_3+1)X\right). 
\end{multline*}

Assuming that $X>ert(C_3+1)=C_5$ we have 
$$m\geq \frac{|\eps^{(i)}|}{4r}\exp(-C_4\log X),$$
with
$$C_4=2^{12t+27}D^2(D-1)^2\left(1+\log(D(D-1))\right)\pi\left(\prod_{k=1}^t h_m(\eta_k)\right)^2.$$

Note that if $K$ is totally real, then all the ratios occurring in $\Lambda$ are real and positive, hence the term $k \pi i$ is absent 
and we may apply Matveev's theorem to $\Lambda$ itself rather than to $r \Lambda$.  Taking $N = 2t$, we then obtain
\begin{equation*}
    C_4=2^{12t+21}D^2(D-1)^2\left(1+\log(D(D-1))\right)\left(\prod_{k=1}^t h_m(\eta_k)\right)^2 .
\end{equation*}

Finally, we assume that
    $$\abs{\delta^{(i)}}=\max\left\{\abs{\eps^{(j)}},\abs{\delta^{(i)}},\abs{\delta^{(j)}}\right\}.$$
Then we obtain
$$\abs{\frac{\delta^{(i)}}{\eps^{(i)}}-1}\leq \frac{\abs{\delta^{(j)}}+\abs{\eps^{(j)}}}{\abs{\eps^{(i)}}}\leq \frac{2m}{\abs{\eps^{(i)}}}.$$
As in the previous cases we get
$$|r\Lambda|\geq \frac{4rm}{\abs{\eps^{(i)}}}$$
with
$$r\Lambda= r(y_1-x_1) \log \eta_1^{(i)} +\dots + r(y_t-x_t) \log \eta_t^{(i)}+\tilde k \pi i.$$
We conclude $\Lambda\neq 0$ because of $\eps^{(i)}\neq \delta^{(i)}$. Thus, we may apply Matveev's theorem with $N=t+1$, and $\alpha_\ell=\eta_\ell^{(i)}$ and $b_\ell=r(y_\ell-x_\ell)$ for $1\leq \ell \leq t$ and $\alpha_{t+1}=-1$ and $b_{t+1}=\tilde k$. Let us note that
$\tilde k<rt(X+Y)<rt(C_3+1)X$. Hence, we may choose $E=rt(C_3+1)X$. Thus by a similar computation as in the previous cases we obtain 
$$\frac{4rm}{\abs{\eps^{(1)}}}\geq \exp(-C_4\log(X))$$
with
$$C_4=2^{6t+27}D^2(D-1)^2\left(1+\log(D(D-1))\right)\pi\left(\prod_{k=1}^t h_m(\eta_k)\right)$$
provided that $X>C_5=ert(C_3+1)X$

In the case that $K$ is totally real we may choose instead
$$C_4=2^{6t+21}D^2(D-1)^2\left(1+\log(D(D-1))\right)\left(\prod_{k=1}^t h_m(\eta_k)\right).$$
\end{proof}

The next proposition is essential in the proof of Theorems \ref{th:gen-bound} and \ref{th:quintic}.

\begin{proposition}\label{prop:matching}
Assume that $\eps$ and $\delta$ are not conjugate algebraic integers. Furthermore, assume that for any indices $1\leq i<j\leq d$ we have that
$$\eps^{(i)}+\delta^{(i)}-\eps^{(j)}-\delta^{(j)} =0$$
implies
$$m={\min}^{(2)}\left(\abs{\eps^{(i)}},\abs{\eps^{(j)}},\abs{\delta^{(i)}},\abs{\delta^{(j)}} \right) \geq \abs{\eps^{(i)}} C$$
for some quantity $C<1$. Then we have $\abs{\eps^{(1)}}\leq C^{-2}$.
\end{proposition}

\begin{proof}
  Let us denote by $j$ the index such that 
  $$\abs{\delta^{(j)}}=\min_{i=1,\dots,d}\left\{\abs{\delta^{(i)}}\right\}.$$

  First, let us assume that $\abs{\delta^{(j)}}<\abs{\eps^{(d)}}$. Assume that $j\neq 1,d$. Then we consider the unit equation
  $$\eps^{(1)}+\delta^{(1)}-\eps^{(j)}-\delta^{(j)} = 0$$
  and deduce by our assumption that $\abs{\eps^{(j)}}\geq m \geq C \abs{\eps^{(1)}}$.
  Next, we consider the unit equation
   $$\eps^{(j)}+\delta^{(j)}-\eps^{(d)}-\delta^{(d)} = 0$$
   and obtain
   $$\abs{\eps^{(d)}}\geq C \abs{\eps^{(j)}} \geq C^2 \abs{\eps^{(1)}}$$

 In case that $j=1$ or $j=d$, we consider
 $$\eps^{(1)}+\delta^{(1)}-\eps^{(d)}-\delta^{(d)} = 0$$
 and obtain $\abs{\eps^{(d)}}\geq C \abs{\eps^{(1)}}$. That is in any case we have $\abs{\eps^{(d)}}\geq C^2 \abs{\eps^{(1)}}$.

 Since $\eps$ is a unit, we have
 $$1=\abs{N_{L/\Q}(\eps)}=\abs{\eps^{(1)}}\cdot \abs{\eps^{(2)}}\cdots \abs{\eps^{(d)}} \geq \abs{\eps^{(1)}}^d C^{2(d-1)},$$
 which implies the proposition in this case.

 Now, we discuss the case that $\abs{\delta^{(j)}}\geq \abs{\eps^{(d)}}$ holds. Let us assume that $j\neq d$ and that $\abs{\eps^{(j)}}>\abs{\delta^{(j)}}\geq \abs{\eps^{(d)}}$. If $j\neq 1$ we consider the unit equation
  $$\eps^{(1)}+\delta^{(1)}-\eps^{(j)}-\delta^{(j)} = 0$$
  and deduce that $\abs{\eps^{(j)}}\geq C \abs{\eps^{(1)}}$. Next we consider the unit equation 
  $$\eps^{(j)}+\delta^{(j)}-\eps^{(d)}-\delta^{(d)}$$
  and deduce that 
  $$\abs{\delta^{(j)}}\geq C \abs{\eps^{(j)}}\geq C^2 \abs{\eps^{(1)}}.$$
  If $j=1$ we consider 
  $$\eps^{(1)}+\delta^{(1)}-\eps^{(d)}-\delta^{(d)} = 0$$
  and deduce that $\abs{\delta^{(j)}}\geq C \abs{\eps^{(1)}}$.
  
  Since $\delta$ is a unit we have
 $$1=\abs{N_{L/\Q}(\delta)}=\abs{\delta^{(1)}}\cdot \abs{\delta^{(2)}}\cdots \abs{\delta^{(d)}} \geq \abs{\eps^{(1)}}^d C^{2d},$$
 which implies the Proposition also in this case.

 We are left to consider the case that $\abs{\delta^{(j)}}\geq \abs{\eps^{(j)}}$. Let us assume that $j\neq 1$. We consider the 
 unit equation
  $$\eps^{(1)}+\delta^{(1)}-\eps^{(j)}-\delta^{(j)}$$
  and deduce that $\abs{\delta^{(j)}}\geq C \abs{\eps^{(1)}}$. Considering the norm of $\delta$ we deduce similarly as in the previous cases that
  $\abs{\eps^{(1)}}\leq C^{-1}$, which implies the statement of the proposition.

  Finally we are left with the case that $j=1$ and $\abs{\delta^{(1)}}\geq \abs{\eps^{(1)}}$. But, since we also assume that $$\abs{\eps^{(1)}}\geq \max_{i=1,\dots,d}\left\{\abs{\delta^{(i)}}\right\}.$$
  we obtain that $\abs{\delta^{(i)}}=\abs{\eps^{(1)}}$ for all $i=1,\dots,d$. Using again that $\delta$ is a unit we deduce that $\abs{\eps^{(1)}}=1<C^{-2}$.
\end{proof}

We combine the results of Lemma \ref{lem:X-eps-bound}, Lemma \ref{lem:matching} and Proposition \ref{prop:matching} with $C=\frac{\exp(-C_4\log X)}{4r}$ and obtain the inequality
\begin{equation}\label{eq:X-ieq}
X<C_2 \log \abs{\eps^{(1)}}<C_2\cdot 2\left(\log(4r) +C_4(\log X)\right).
\end{equation}
By an application of Lemma \ref{lem:PdW} we obtain $X<C_6$
with
$$C_6=4C_2\log(4r)+4C_2C_4\log (2C_2C_4).$$
By Lemma \ref{lem:X-eps-bound} we also deduce that $Y<C_7=C_3C_6$. Moreover, we note that $C_3>1$, hence $X,Y<C_7$ and we obtain by Lemma \ref{lem:X-eps-bound}
$$\max\{\log |\eps^{(i)}|,\log |\delta^{(i)}|\}<C_7/C_1.$$
Since for a unit $\alpha\in \ord_K^*$ we deduce from the definition of height that
$$h(\alpha)<\max_{1\leq i\leq d}\{\log |\alpha^{(i)}|\},$$
we deduce that $h(\eps),h(\delta)<C_7/C_1$. Since $n=\eps +\delta$, we also obtain 
$$\log  |n|<\log \left(2\max\{\log |\eps^{(i)}|,\log |\delta^{(i)}|\}\right)<2C_7/C_1=:C_8.$$

Let us record, what we have proved so far:

\begin{proposition}\label{prop:height-bound}
Let $K$ be a given number field and let $n\neq 0$ be an integer and $\eps,\delta \in \ord_K^*$ such that $\eps,\delta$ is not a pair of conjugate algebraic integers. Assume that $n=\eps+\delta$, then
 $$\max\{\log |n|,h(\eps),h(\delta)\}\leq C_8.$$
\end{proposition}

Note that Theorem \ref{th:gen-bound} is an immediate consequence of Proposition \ref{prop:height-bound}. That is to compute $N_K^-$ we are left to consider the Diophantine equation
\begin{equation}\label{eq:conj} n=\eps^{(1)}+\eps^{(i)}
\end{equation}
for some $1\leq i\leq d$, with $d=[\Q(\eps):\Q]>2$.

\section{Restricting Galois groups} \label{sec:galoisgroups}

In this section, we prove constraints on the possible number fields for which equation~\eqref{eq:conj} has a solution. In the following we denote be $\Gal(L/\Q)$ the Galois group of the Galois closure of $L/\Q$.

\begin{lemma} \label{lem:galoisgroup}
    Let $L$ be a number field of degree $d$ and assume there exists a unit $\eps \in \ord_L^*$ such that $L = \Q(\eps)$ and $n - \eps$ is conjugate to $\eps$ for some nonzero rational integer $n$.  
    Then either $d=1$, or $d$ is even and $\Gal(L/\Q)$ is a subgroup of the wreath product $C_2 \wr S_{d/2}$ (as a degree $d$ permutation group).
\end{lemma}

\begin{proof}
    Let us assume $n = \eps^{(1)}+\eps^{(i)}$ for some $1 \leq i \leq d$. If $i = 1$, then $\eps^{(1)}$ must be rational, and thus $d = 1$.

    Now assume $i \not = 1$, and thus $\eps^{(1)} \not = \eps^{(i)}$. By considering the various embeddings of $\eps$, we obtain that every index $j$ is contained in one of the representations of $n$
   $$\eps^{(j)}+\eps^{(k)}=n.$$
   Assume that there is another representation 
   $$\eps^{(j)}+\eps^{(\tilde k)}=n$$
   with $k \neq \tilde k$. Subtracting the two equations gives $\eps^{(k)} = \eps^{(\tilde k)}$, a contradiction. Hence the index $k$ is uniquely determined by $j$.  That is the pair of indices $(i_1,i_2)$ for which
   $$\eps^{(i_1)}+\eps^{(i_2)}=n$$
   form a partition of the set of all indices $\{i\: :\: 1\leq i \leq d\}$ into subsets of size two, which implies that $d$ is even. 
   Let this partition be $\{ (i_1, i_2), (i_3, i_4), \dots, (i_{d-1}, i_d) \}$, that is, we have
    \begin{equation} \label{eq:pairing}
        \eps^{(i_\ell)} + \eps^{(i_{\ell+1})} = n
    \end{equation}
   for every odd $\ell < d$.  Now let $\sigma \in \Gal(L/\Q)$.  We can consider $\sigma$ as a degree $d$ transitive permutation group where (by a slight abuse of notation) we can consider $G$ as a subgroup of $S_d$, where we define $\sigma(\ell)$ as $\eps^{(i_{\sigma(\ell)})} := \sigma ( \eps^{(i_\ell)} )$.  From \eqref{eq:pairing}, we have that 
   \begin{equation*}
       \eps^{(i_{\sigma(\ell)})} + \eps^{(i_{\sigma(\ell+1)})} = n
   \end{equation*}
   for all $\sigma \in \Gal(L/\Q)$. However, this implies that, for any $\sigma \in \Gal(L/\Q)$,  the set $\{ \sigma(\ell), \sigma(\ell+1) \}$ equals the set $\{ \ell', \ell' + 1 \} $, for some odd $\ell' < d$.  This therefore constrains the possible action of $\sigma \in \Gal(L/\Q)$.  Indeed, let $B_i$ denote the set $\{2i - 1, 2i\}$, for each $1 \leq i \leq d/2$. Then we have that
    \begin{equation*}
        \Gal(L/\Q) \subseteq \{ \sigma \in S_d \;|\; \text{for all } i \leq d/2, \; \sigma ( B_i) = B_j \text{ for some }  j \leq d/2 \} .
    \end{equation*}
    The group on the right hand side is precisely, by definition, the degree $d$ wreath product $C_2 \wr S_{d/2}$, i.e. the semidirect product $C_2^{d/2} \rtimes S_{d/2}$ where $S_{d/2}$ acts in the natural way on $d/2$ copies of $C_2$.  Therefore $\Gal(L/\Q) \subseteq C_2 \wr S_{d/2}$.
\end{proof}

As a consequence of the above proof, we can also show that, if $L$ is not rational, then $L$ must contain a subfield of index 2, and thus cannot be primitive.  We provide a short explicit proof of this below.

\begin{lemma}\label{lem:subfield}
    Let $L$ be a non-rational number field as in Lemma~\ref{lem:galoisgroup}.  Then $L$ contains a proper subfield of index 2 .
\end{lemma}

\begin{proof}
    We have that $n = \eps^{(1)}+\eps^{(i)}$ for some $2 \leq i \leq d$.  Let $f(x) \in \Z[x]$ be the minimal polynomial of $\eps^{(1)}$ over $\Q$. As $\eps \not \in \Q$, $f(x)$ has degree at least 2. As $n - \eps^{(1)}$ is conjugate to $\eps^{(1)}$, we have that $f(x) = f(n-x)$.
    Now let $g(x)$ be the minimal polynomial of $\eps^{(1)} - n/2$, i.e. $g(x) := f(x + n/2)$.  Thus, we have that
    \begin{equation*}
        g(x) = f(x + n/2) = f(n - x - n/2) = g(-x),
    \end{equation*}
    which implies $g(x)$ is an even polynomial, and thus $g(x) = h(x^2)$ for some degree $d/2$ polynomial $h \in \Z[x]$.  In particular, the element $(\eps^{(1)} - n/2)^2$ has minimal polynomial of degree $d/2$ over $\Q$ and thus generates a subfield of index 2 in $L$.
\end{proof}

Note that a combination of Proposition \ref{prop:height-bound}, Lemma \ref{lem:galoisgroup} and Lemma \ref{lem:subfield} yields Corollary~\ref{cor:galgroup}.

\medskip
\noindent \textbf{Example.} Let us consider a degree 4 number field $L$ which satisfies the conditions of Lemma~\ref{lem:galoisgroup}.  Then $\Gal(L/\Q)$ must be contained in the wreath product $C_2 \wr S_2$, and thus must be one of the following three imprimitive degree 4 transitive permutation groups: $C_4$, $C_2 \times C_2$, or $D_4$.  In particular $\Gal(L/\Q)$ cannot be either of the two primitive degree 4 groups $A_4$ or $S_4$.

We remark that the constraints on the Galois group $\Gal(L/\Q)$ in Lemma~\ref{lem:galoisgroup} cannot be improved, at least in the degree 4 case.  In particular, for each of the three imprimitive degree 4 Galois groups, we can find examples of number fields $L$ containing a unit $\eps \in \ord_L^*$ which is conjugate to $1 - \eps$, as shown in Table~\ref{tab:galgroup}.

\begin{table}[h]
\centering
\caption{Examples of imprimitive degree 4 fields $L$ containing a primitive unit $\eps \in \ord_L^*$ such that $1 - \eps$ is a Galois conjugate of $\eps$.}
\label{tab:galgroup}
\begin{tabular}{@{}ccc@{}}
\toprule
\textbf{LMFDB label of $\Q(\eps)$} & \textbf{Minimal polynomial of $\eps$} & \textbf{Galois group} 
\\ \midrule

4.0.125.1 &  $x^4 - 2x^3 + 4x^2 - 3x + 1$ &  $C_4$ (as 4T1) \\[2mm]
4.0.144.1 &  $x^4 - 2x^3 + 5x^2 - 4x + 1$   & $C_2 \times C_2$ (as 4T2) \\[2mm]
4.0.117.1 & $x^4 - 2x^3 + 2x^2 - x + 1$ & $D_4$ (as 4T3) \\ 
\bottomrule
\end{tabular}
\end{table}

It is worth noting that, although an effective algorithm to compute $\NK^-$ for every degree 4 field remains out of reach, we have that the results of Bhargava \cite[Theorem~4]{Bhargava2005} and Cohen--Diaz y Diaz--Olivier \cite[Corollary~1.4]{CDO2002} imply that about 83\% of quartic fields (ordered by discriminant) have Galois group $S_4$, with the remaining 17\% having Galois group $D_4$. Consequently, we can thus apply Theorem~\ref{th:gen-bound} to 83\% of quartic fields (ordered by discriminant).

\section{Example: A quintic field} \label{sec:quintic_example}

Let $K=\Q(\alpha)$, where $\alpha$ is a root of $X^5-5X^3-X^2+5X+1$.  In this section, we derive an explicit bound for the constant $C$ in Theorem~\ref{th:gen-bound} applied to the field $K$, and use LLL-reduction to give a proof of Theorem~\ref{th:quintic}.

Since $[K:\Q]=5$ is odd, we find a first bound for 
$$\max_{i=1,2,3,4,5}\left\{\left|\eps^{(i)}\right|,\left|\delta^{(i)}\right|\right\}$$
for a representation $n=\eps+\delta$ of an integer $n$, by computing the bounds provided by Lemma \ref{lem:matching} and Proposition \ref{prop:height-bound}. That is we aim to compute the constants $C_1,\dots,C_8$.

Using SageMath \cite{sagemath}, we obtain that a fundamental system of units $\eta_1, \dots, \eta_4$ for $K$ can be given by
\begin{align}
\eta_1 &= \alpha^2 - 2,& \eta_2&= -\alpha^4 + \alpha^3 + 4\alpha^2 - 2\alpha - 3, \label{eq:etas} \\
\eta_3&= -\alpha^4 + \alpha^3 + 3\alpha^2 - 2\alpha - 1,& \eta_4&= \alpha^4 - \alpha^3 - 4\alpha^2 + 2\alpha + 2. \nonumber
\end{align}
Computing the real embeddings from $K\to \R$ we obtain the matrix
$$R=\left(\log|\eta_i^{(j)}|\right)=\left(\begin{array}{rrrr}
-0.1300006 & 0.2316942 & 0.7602585 & -1.344254 \\
-0.9951165 & 0.2892002 & 0.5335882 & 0.8481656 \\
0.6729581 & 0.8960184 & -0.7137523 & 0.3714459 \\
-0.3246052 & -1.080787 & -0.9584567 & -0.4144984 \\
0.7767642 & -0.3361260 & 0.3783623 & 0.5391408
\end{array}\right).$$
Note that this matrix $R$ represents a precision of only $30$-bits. We made our computations with $1000$-bit precision and $10\,000$-bit precision in case of LLL-reduction. For reason of space we present only the $30$-bit precision.

From this matrix it is easy to compute the constants
\begin{align*}
    C_1 &= \frac 1{t\max_{i=1,\dots,t}\left\{\log\abs{\eta^{(1)}_i}\right\}}=0.27901212 ,
    \\[2mm]
    C_2 &= t\frac{\max_{1\leq i,j\leq t}\left\{\abs{\det R_{i,j}}\right\}}{\Reg_K}=1.46930513 , \\[2mm]
    C_3 &= \frac{C_2}{C_1}=5.26609783.
\end{align*}

Next, we aim to compute the constants $C_4$ and $C_5$. That is, we have to compute the modified heights $h_m(\eps_i)$ first. In particular we obtain:
\begin{gather*}
h_m(\eps_1)=7.24861112, \qquad h_m(\eps_2)=7.08456406,\\
h_m(\eps_3)=8.36104504,\qquad  h_m(\eps_4)=8.79376172, 
\end{gather*}
Let us note that in the computation of $C_4$ we have to distinguish between the three cases considered in the proof of Lemma \ref{lem:matching}.

Let us note that the largest constants appear in the second case. There we obtain
$$C_4<2.257\cdot 10^{30}$$
provided that $X<C_5=68.133$. Now, we solve inequality \eqref{eq:X-ieq} directly instead of using Lemma \ref{lem:PdW} and obtain
$$X<C_6=1.008 \cdot 10^{33}.$$
Therefore we obtain
$$
\log|\eps^{(1)}|=\max\{\log|\eps^{(i)}|,\log|\delta^{(i)}|\}< C_7/C_1 = C_8/2 < 
1.91\cdot 10^{34}. 
$$

Since this bound is far too huge for a brute force search, we use the LLL-reduction method to get a much lower bound for $\max\{\log|\eps^{(i)}|,\log|\delta^{(i)}|\}$.

We want to find a small upper bound $B$ for $\max_{1\leq i\leq 5} \{\log|\eps^{(i)}|\}$. From the previous section we have $B=1.91\cdot 10^{34}$. We want to apply Lemma \ref{lem:real-reduce}. That is, we consider for $1\leq i,j\leq 5$ with $i\neq j$ the linear forms
\begin{align} \label{eq:L1}
\Lambda_1^{(i,j)}&=x_1 \log \abs{\frac{\eta_1^{(j)}}{\eta_1^{(i)}}} +x_2 \log \abs{\frac{\eta_2^{(j)}}{\eta_2^{(i)}}} +x_3 \log \abs{\frac{\eta_3^{(j)}}{\eta_3^{(i)}}} +x_4\log \abs{\frac{\eta_4^{(j)}}{\eta_4^{(i)}}} , \\
\label{eq:L2}
\Lambda_2^{(i,j)}&=y_1 \log \abs{\eta_1^{(j)}} +y_2 \log \abs{\eta_2^{(j)}} +y_3 \log \abs{\eta_3^{(j)}}+y_4\log \abs{\eta_4^{(j)}} \\
\nonumber &\qquad-x_1 \log \abs{\eta_1^{(i)}} -x_2 \log \abs{\eta_2^{(i)}} -x_3 \log \abs{\eta_3^{(i)}}-x_4\log \abs{\eta_4^{(i)}} , \\
\label{eq:L3}
\Lambda_3^{(i)}&=(y_1-x_1) \log \abs{\eta_1^{(i)}} +(y_2-x_2) \log \abs{\eta_2^{(i)}} +(y_3-x_3) \log \abs{\eta_3^{(i)}}\\
\nonumber &\qquad +(y_4-x_4)\log \abs{\eta_4^{(i)}} .
\end{align}

Let us note that $\max\{\log|\eps_i^{(j)}|\}\leq B$ is equivalent to 
\begin{equation}\label{eq:LinProgPolytop}
R\cdot \mathbf x\leq \mathbf B,
\end{equation}
with $\mathbf x=(x_1,x_2,x_3,x_4)^T\in \Z^4$ and $\mathbf B=(B,B,B,B,B)^T$. Let us denote the set of solutions to \eqref{eq:LinProgPolytop} by $M$. In view of Lemma \ref{lem:real-reduce} we want to find
$$T=\max_{\mathbf x\in M} \{ |x_1|+|x_2|+|x_3|+|x_4|\}$$
and
$$S=\max_{\mathbf x\in M} \{x_1^2+x_2^2+x_3^2\}.$$
For the application of Lemma \ref{lem:real-reduce} to \eqref{eq:L2} we also have to compute
$$S'=\max_{\mathbf x\in M} \{x_1^2+x_2^2+x_3^2+x_4^2\}.$$
Let us note that the polytope described by \eqref{eq:LinProgPolytop} is a simplex and therefore convex. That is the functions $x_1^2+x_2^2+x_3^2$ and $x_1^2+x_2^2+x_3^2+x_4^2$ take their maximal value in one of the vertices of the polytope. That is we obtain
$$
S\leq 13.8511646 \cdot B^4\qquad\text{and}\qquad S'\leq 15.38008419 \cdot B^4.
$$
To determine $T$ we want to maximize the $16$ linear functions $\pm x_1\pm x_2 \pm x_3 \pm x_4$ subject to \eqref{eq:LinProgPolytop} and take the maximal value of the $16$ maximal values. This can be easily done in Sage \cite{sagemath}. In the case of $B=1.91\cdot 10^{34}$ we obtain $T=1.4581 \cdot 10^{35}$.

With these computations we apply Lemma \ref{lem:real-reduce} to each of the linear forms $\Lambda_1^{(i,j)}$ using the parameters $S_1=S$ and $T_1=\frac{1+T}2$ and $\tilde C_1=10^{150}$. With these parameters we obtain that either $x_1=x_2=x_3=x_4=0$ or $C_0<|\Lambda_1^{(i,j)}|$ for some $C_0$. For the linear forms $\Lambda_2^{(i,j)}$ we use the parameters $S_2=S+S'$ and $T_2=\frac{1+2T}2$ and $\tilde C_2=10^{290}$ to obtain a lower bound for $|\Lambda_2^{(i,j)}|$. In the case of the linear forms $\Lambda_3^{(i)}$ we use the parameters $S_3=4S$ and $T_2=\frac{1+2T}2$ (note that $|y_i-x_i|\leq 2|x_i|$) and $\tilde C_2=10^{150}$ to obtain a lower bound for $|\Lambda_3^{(i)}|$.
This yields a lower bound for $C$ in Lemma \ref{lem:matching} and by Proposition \ref{prop:matching} we obtain a smaller bound $B$ such that 
$$\max_{1\leq i\leq 5} \{\log|\eps^{(i)}|\}\leq B.$$
Indeed we obtain $B=1172.6$ with the above mentioned choices.

We repeat this process with new choices for $\tilde C_1, \tilde C_2$ and $\tilde C_3$ six times more to obtain an even smaller bound $B$. For the choices for  $\tilde C_1, \tilde C_2$ and $\tilde C_3$ and the respectively new bound $B$ we refer to Table \ref{tab:LLL-red}.

\begin{table}[h]
\caption{Choice of $\tilde C_1, \tilde C_2$ and $\tilde C_3$ and the resulting bound $B$.}\label{tab:LLL-red}

\begin{tabular}{ccccc}
\toprule
\textbf{Steps} & \textbf{Choice for $\tilde C_1$} & \textbf{Choice for $\tilde C_2$} & \textbf{Choice for $\tilde C_3$} & \textbf{New bound $B$} \\\midrule
1 & $10^{150}$ & $10^{290}$ & $10^{150}$ &  $1172.6$ \\ 
2 & $10^{18}$ & $10^{34}$ & $10^{18}$ & $146.32$ \\ 
3 & $10^{14}$ & $9\cdot 10^{26}$ & $10^{14}$ & $115.8$  \\ 
4 & $10^{14}$ & $10^{26}$ & $10^{14}$ & $112.29$  \\ 
5 & $10^{14}$ & $4\cdot 10^{25}$ & $10^{14}$ & $111.63$  \\ 
6 & $10^{14}$ & $4\cdot 10^{25}$ & $10^{14}$ & $111.4971$  \\ 
7 & $10^{14}$ & $4\cdot 10^{25}$ & $10^{14}$ & $111.49$  \\ 
\bottomrule
\end{tabular}
\end{table}

That is, we have proved so far.

\begin{lemma}\label{lem:bound-after-LLL}
Assume that there exist units $\eps,\delta\in \ord_K$ with $n=\eps+\delta$. Then we have
$$\max_{1\leq i\leq 5}\left\{\log \abs{\eps^{(i)}}\right\},\max_{1\leq i\leq 5}\left\{\log \abs{\delta^{(i)}}\right\}\leq 111.49.$$
\end{lemma}

Let us note that one can easily show that the lemma implies $n \leq 2.45 \cdot 10^{65}$. However, we will not consider this bound for our next step.

Consider an integer quadruple $(x_1,x_2,x_3,x_4)\in \Z^4$ that satisfies \eqref{eq:LinProgPolytop} with $B=111.49$. Let $\eps=\pm \eta_1^{x_1}\eta_2^{x_2}\eta_3^{x_3}\eta_4^{x_4}$, then we want to determine whether there exists a positive integer $n$ and a unit $\delta\in \ord_K^*$ such that $\eps+\delta=n$. Let us note that if $n=\eps+\delta$ then we also have $-n=-\eps-\delta$ and therefore it is sufficient to only consider those units $\eps$ with $\eps=\eta_1^{x_1}\eta_2^{x_2}\eta_3^{x_3}\eta_4^{x_4}$.

Let us assume that $|n|>1000$ and choose the index $j$ such that
$$\abs{n-\eps^{(j)}}=\min_{1\leq i\leq 5}\left\{\abs{n-\eps^{(i)}}\right\}<1.$$
By exchanging the roles of $\eps$ and $\delta$ we may also assume that
$$\abs{n-\eps^{(j)}}\leq \min_{1\leq i\leq 5}\left\{\abs{n-\delta^{(i)}}\right\}.$$
Assume for the moment that $\abs{n-\eps^{(j)}}\geq 1/2$, then we obtain for all $i=1,\dots,5$ the lower bound
$$\abs{\eps^{(i)}}=\abs{n-\delta^{(i)}}\geq \abs{n-\eps^{(j)}}\geq 1/2$$
and therefore we deduce by the product formula that $\max_{1\leq i\leq 5}\{\abs{\eps^{(i)}}\}\leq 16$. But, this implies 
$$|n|\leq \abs{n-\eps^{(j)}}+\abs{\eps^{(j)}}< 1+16= 17,$$
which contradicts our assumption that $|n|> 1000$. Therefore, we may assume that 
$$\abs{n-\eps^{(j)}}< \frac 12.$$
This also implies $n=\left\lceil \eps^{(j)} \right\rfloor$, where $\lceil x\rfloor$ denotes the nearest integer to $x$.

Next, we have for $i\neq j$ the following inequality:
$$2\abs{n-\eps^{(i)}}\geq \abs{n-\eps^{(i)}}+\abs{n-\eps^{(j)}}\geq\abs{\eps^{(i)}-\eps^{(j)}}. $$
Therefore we obtain 
\begin{equation}\label{eq:Test-eps}
\abs{n-\eps^{(j)}}=\prod_{i\neq j} \frac{1}{\abs{n-\eps^{(i)}}}\leq \frac{16}{\prod_{i\neq j}\abs{\eps^{(i)}-\eps^{(j)}}}.
\end{equation}

Let $M\subseteq \Z^4$ be the set of integer quadruples $(x_1,x_2,x_3,x_4)\in \Z^4$ that satisfies \eqref{eq:LinProgPolytop} with $B=111.49$. For a given quadruple $\mathbf x=(x_1,x_2,x_3,x_4)\in M$ we proceed as follows:
\begin{enumerate}
\item Compute $\eps=\eta_1^{x_1}\eta_2^{x_2}\eta_3^{x_3}\eta_4^{x_4}$ and compute all its embeddings $\eps^{(i)}$ to a high enough precision. That is the numerical computation of $\left\lceil \eps^{(j)} \right\rfloor$ is correct. Since we have $\max_{1\leq i\leq 5}\left\{\log \abs{\eps^{(i)}}\right\}<111.49$ a $10\,000$-bit precision will be sufficient.
\item Consider all indices $j$ such that $n=\left\lceil \eps^{(j)} \right\rfloor>1000$.
\item Compute $\beta=\abs{n-\eps^{(j)}}$ and 
$$u=\frac{16}{\prod_{i\neq j}\abs{\eps^{(i)}-\eps^{(j)}}}.$$
\item If $\beta> u$, discard $\eps^{(i)}$, otherwise add $n$ to the list of possible integers that admit a representation of the form $n=\eps+\delta$, with $\eps, \delta\in \ord_K^*$.
\end{enumerate}

Let us discuss how to numerate the quadruples $\mathbf x=(x_1,x_2,x_3,x_4)\in M$. Therefore we note thatby maximizing $x_1$ and $-x_1$ subject to \eqref{eq:LinProgPolytop} we get an upper bound $u_1$ and a lower bounds $l_1$ for $x_1$.

Now, for each $x_1$ with $l_1\leq x_1\leq u_1$ we consider the polytope given by
\begin{equation}\label{eq:Lin-Prog-2}
R_2 \cdot \left(\begin{array}{c} x_2 \\ x_3 \\ x_4 \end{array}\right) \leq \mathbf B_{x_1}
\end{equation}
where $R_2$ is the matrix which we get from $R$ by the deleting the first column and $B_{x_1}$ is the vector
$$
\mathbf B_{x_1}=\left(\begin{array}{c} B\\ B\\B \\ B\\ B\end{array}\right)-x_1\left(\begin{array}{c} \eta_1^{(1)}\\ \eta_1^{(2)}\\ \eta_1^{(3)} \\  \eta_1^{(4)}\\  \eta_1^{(5)}\end{array}\right).
$$
Maximizing $x_2$ and $-x_2$ subject to \eqref{eq:Lin-Prog-2} yields upper and lower bounds for $x_2$ for each $x_1$.

Similarly we obtain upper and lower bounds for $x_3$ for given $x_1$ and $x_2$. And finally for given $x_1,x_2$ and $x_3$ we obtain also upper and lower bounds for $x_4$.

The computation were carried out for all $\mathbf{x}=(x_1,x_2,x_3,x_4)\in M$ out on a standard desktop computer; the total running time was approximately 23 days on a single core. Let us note that $|M|=1\,676\,902\,296 \simeq 1.68 \cdot 10^9$.

Since for no instance we found a quadruple $\mathbf{x}=(x_1,x_2,x_3,x_4)\in M$ that satisfies inequality \eqref{eq:Test-eps}, we have the following:

\begin{lemma}\label{lem:only-small-n}
Assume that there exist units $\eps,\delta\in \ord_K^*$ and a rational integer $n$ such that $n=\eps+\delta$. Then we have $|n| \leq 1000$.
\end{lemma}

Since we only have to consider positive $n$ we solve the $1000$ unit equations
\begin{equation} \label{eq:quintic}
    n=\eps+\delta
\end{equation}
for each $1 \leq n \leq 1000$.  This can easily be done in Magma \cite{magma} with the function \texttt{UnitEquation}, implemented using Wildanger's method \cite{Wildanger}. This took roughly 3 CPU-hours running on a single core.

Solutions to the unit equation $\eps + \delta = n$ were found for each $n = 1$, $2$, $3$, $4$, $5$, $6$, $7$, $10$, $11$, $15$, $251$, for a total of $\numquinticsols$ non-equivalent pairs of solutions $(\eps, \delta)$.  A summary of the number of solutions found for each $n$ is given in Table~ \ref{tab:quintic_summary}.

\begin{table}[h]
\centering
\caption{For each $n \in \NK$, we tabulate the number of non-equivalent solutions to $\eps + \delta = n$ over the smallest totally real degree 5 quintic field $K$.}
\label{tab:quintic_summary}
\begin{tabular}{@{}c@{\hskip 5mm}ccccccccccc@{}}
\toprule
$n$ & \textbf{1} & \textbf{2} & \textbf{3} & \textbf{4} & \textbf{5} & \textbf{6} & \textbf{7} & \textbf{10} & \textbf{11} & \textbf{15} & \textbf{251} \\ \midrule
\begin{tabular}[c]{@{}c@{}}Number of non-equivalent \\ pairs $(\eps, \delta)$ such that $\eps + \delta = n$. \end{tabular} & 168 & 44 & 16 & 8 & 9 & 2 & 4 & 1 & 1 & 1 & 1 \\ \bottomrule
\end{tabular}
\end{table}

A full tabulation of all solutions found is given in Table~\ref{tab:quinticsolutions}.  In particular, we have that $X = \max_{i = 1, 2, 3, 4} |x_i| \leq 10$ for all possible solutions, with the largest solution to $\eps + \delta = n$ being 
\begin{equation*}
    \eps = + \eta_1^{-4} \eta_2^{10} \eta_3^0 \eta_4^{-2}, \quad \delta = - \eta_1^{-1} \eta_2^{-8} \eta_3^1 \eta_4^6, \quad \text{and} \quad n = 251 .
\end{equation*}
This completes the proof of Theorem~\ref{th:quintic}.

\bibliographystyle{abbrv}
\bibliography{NK}

\begin{longtable}{@{}c@{\hskip 5mm}c@{}}
\caption{List of all solutions to $\eps + \delta = n$ over the totally real degree~5 field $K$ with defining polynomial $X^5 - 5X^3 - X^2 + 5X + 1$, given up to equivalence. All solutions were computed with Magma~\cite{magma}. Here, $\eta_1, \eta_2, \eta_3, \eta_4$ are a fundamental system of units for $K$, given in (\ref{eq:etas}).}
\label{tab:quinticsolutions} \\

\toprule
\textbf{$n$} & \textbf{Tuples $(\pm, x_1, x_2, x_3, x_4)$ where $\eps = \pm \eta_1^{x_1} \eta_2^{x_2} \eta_3^{x_3} \eta_4^{x_4}$ satisfies (\ref{eq:quintic})} \\* \midrule
\endfirsthead
\multicolumn{2}{c}%
{{ %
Table \thetable\ continued from previous page}} \\
\toprule
\textbf{$n$} & \textbf{Tuples $(\pm, x_1, x_2, x_3, x_4)$ where $\eps = \pm \eta_1^{x_1} \eta_2^{x_2} \eta_3^{x_3} \eta_4^{x_4}$ satisfies (\ref{eq:quintic})}  \\* \midrule
\endhead
\bottomrule
\endfoot
\endlastfoot

$\mathbf{1}$ &  \begin{tabular}[c]{@{}c@{}}

\\[-2mm]
 $(-,0,0,0,1), \, (-,0,0,0,-1), \, (-,0,0,-1,0), \, (-,0,0,1,1), \, (-,0,0,1,-1), $ \\ 
 $(-,0,0,1,2), \, (-,0,0,-1,-2), \, (-,0,0,2,-1), \, (-,0,0,-2,1), \, (+,0,1,0,0), $ \\ 
 $(+,0,-1,0,0), \, (-,0,-1,0,0), \, (-,0,1,0,1), \, (-,0,-1,0,-1), \, (-,0,-1,0,2), $ \\ 
 $(-,0,1,1,0), \, (+,0,1,-1,0), \, (-,0,1,-1,0), \, (+,0,-1,1,0), \, (-,0,-1,1,0), $ \\ 
 $(-,0,-1,-1,0), \, (-,0,1,1,-1), \, (+,0,1,-1,1), \, (-,0,-1,-1,1), \, (-,0,1,-2,1), $ \\ 
 $(+,0,1,-2,2), \, (-,0,1,-2,2), \, (+,0,-1,2,-2), \, (-,0,-1,2,-2), \, (+,0,-1,-5,-3), $ \\ 
 $(+,0,-2,0,0), \, (+,0,2,-1,0), \, (+,0,-2,1,0), \, (+,0,2,-2,0), \, (+,0,-2,-2,0), $ \\ 
 $(+,0,2,2,-1), \, (+,0,-2,-2,1), \, (+,0,-2,4,-4), \, (-,0,3,0,-1), \, (-,0,-3,0,1), $ \\ 
 $(+,0,3,-1,0), \, (+,0,-3,1,0), \, (+,0,3,1,1), \, (+,0,-3,-1,-1), \, (-,0,4,-3,-1), $ \\ 
 $(+,1,0,0,0), \, (-,1,0,0,0), \, (+,-1,0,0,0), \, (-,-1,0,0,0), \, (+,1,0,0,1), $ \\ 
 $(+,1,0,0,-1), \, (+,-1,0,0,1), \, (+,-1,0,0,-1), \, (-,1,0,0,-2), \, (-,-1,0,0,2), $ \\ 
 $(+,1,0,1,0), \, (-,-1,0,1,0), \, (+,-1,0,-1,0), \, (+,1,0,-1,-1), \, (+,-1,0,1,-1), $ \\ 
 $(-,1,0,1,-2), \, (-,-1,0,1,-2), \, (-,-1,0,-1,2), \, (+,1,0,1,-3), \, (+,-1,0,-1,3), $ \\ 
 $(-,1,0,2,0), \, (+,1,0,-2,0), \, (+,-1,0,2,0), \, (-,-1,0,-2,0), \, (+,-1,0,-2,-1), $ \\ 
 $(-,1,0,2,2), \, (-,-1,0,-2,-2), \, (-,1,1,0,0), \, (-,-1,-1,0,0), \, (-,1,1,0,-1), $ \\ 
 $(+,1,-1,0,1), \, (+,-1,1,0,-1), \, (-,-1,-1,0,1), \, (+,1,-1,0,2), \, (+,1,-1,0,-2), $ \\ 
 $(+,-1,1,0,2), \, (+,-1,1,0,-2), \, (-,1,1,-1,0), \, (+,1,-1,1,0), \, (-,-1,-1,1,0), $ \\ 
 $(-,-1,-1,-1,0), \, (-,1,1,-1,-1), \, (+,-1,1,1,1), \, (-,-1,-1,1,1), \, (-,-1,-1,2,-1), $ \\ 
 $(+,1,-1,4,-1), \, (+,-1,1,-4,1), \, (+,1,2,0,0), \, (+,-1,-2,0,0), \, (-,1,-2,0,2), $ \\ 
 $(-,1,-2,0,3), \, (-,-1,2,0,-3), \, (-,1,2,0,-6), \, (-,-1,-2,0,6), \, (-,-1,-2,-1,0), $ \\ 
 $(+,1,2,1,-1), \, (+,1,2,-1,1), \, (-,1,-2,-1,-1), \, (-,1,-2,-1,-2), \, (-,-1,2,1,2), $ \\ 
 $(-,1,-2,3,0), \, (-,-1,2,-3,0), \, (-,1,2,4,-3), \, (+,1,2,-4,3), \, (+,-1,-2,4,-3), $ \\ 
 $(-,-1,-2,-4,3), \, (+,1,3,-1,1), \, (+,-1,-3,-1,-1), \, (-,1,4,-1,0), \, (-,-1,4,-4,0), $ \\ 
 $(+,1,-6,2,0), \, (+,-1,6,-2,0), \, (+,-2,0,0,0), \, (+,-2,0,0,2), \, (+,2,0,2,0), $ \\ 
 $(+,2,0,2,1), \, (+,-2,0,-2,-1), \, (-,2,1,0,0), \, (-,2,-1,0,0), \, (-,-2,1,0,0), $ \\ 
 $(-,-2,-1,0,0), \, (+,2,-1,0,1), \, (-,-2,1,0,1), \, (+,-2,1,0,-1), \, (+,2,1,0,2), $ \\ 
 $(+,-2,-1,0,-2), \, (+,2,-1,1,0), \, (+,2,-1,1,1), \, (+,2,-1,1,2), \, (-,-2,-1,1,2), $ \\ 
 $(+,2,1,3,0), \, (-,-2,2,0,-3), \, (-,2,-2,1,1), \, (-,-2,-2,-1,-1), \, (-,2,-2,1,3), $ \\ 
 $(-,2,3,0,-1), \, (-,-2,-3,0,1), \, (-,3,0,-1,0), \, (-,-3,0,1,0), \, (+,3,0,1,-5), $ \\ 
 $(-,-3,0,3,2), \, (+,3,0,4,1), \, (+,-3,0,-4,-1), \, (-,-3,1,-1,1), \, (+,3,1,-1,-3), $ \\ 
 $(+,-3,-1,1,3), \, (-,3,-1,3,2), \, (+,3,-2,1,1), \, (+,3,2,-3,-1), \, (-,-4,0,1,1), $ \\ 
 $(+,4,-1,1,0), \, (+,-4,1,-1,0), \, (-,4,-1,2,-1), \, (-,-4,1,-2,1), \, (+,4,-1,3,1), $ \\ 
 $(-,4,-3,0,0), \, (-,-4,3,0,0), \, (-,-4,-3,2,1), \, (-,4,-4,1,5), \, (+,-5,2,-1,1), $ \\ 
 $(-,5,-3,0,-1), \, (-,-5,3,0,1), \, (-,6,-3,2,0). $ \\[2mm]

\end{tabular}  \\ %

$\mathbf{2}$ &  \begin{tabular}[c]{@{}c@{}}

\\[-2mm]
 $(+,0,0,0,0), \, (-,0,0,0,1), \, (+,0,0,0,2), \, (-,0,0,1,0), \, (-,0,0,-1,0), $ \\ 
 $(-,0,0,-1,1), \, (-,0,0,-1,2), \, (+,0,0,2,0), \, (+,0,0,-2,2), \, (+,0,0,-2,-2), $ \\ 
 $(-,0,-1,0,1), \, (+,0,1,1,0), \, (-,0,-1,1,0), \, (+,0,2,0,0), \, (+,0,-2,0,0), $ \\ 
 $(-,0,2,1,-1), \, (+,0,2,2,1), \, (-,0,3,0,0), \, (-,1,0,0,0), \, (+,-1,0,0,-1), $ \\ 
 $(-,1,0,0,2), \, (+,-1,0,1,-1), \, (-,-1,-1,0,0), \, (-,1,-1,-1,0), \, (-,-1,1,-1,0), $ \\ 
 $(+,1,1,1,1), \, (-,-1,-1,3,0), \, (+,1,-2,-2,0), \, (+,1,2,3,-1), \, (+,1,3,1,1), $ \\ 
 $(-,-1,3,1,-2), \, (-,1,-3,2,0), \, (+,-1,-3,3,0), \, (-,1,4,1,0), \, (+,-2,0,0,0), $ \\ 
 $(-,2,0,1,0), \, (+,2,0,2,2), \, (-,2,0,2,-3), \, (-,2,1,0,-1), \, (+,-2,2,0,0), $ \\ 
 $(+,2,-2,0,2), \, (-,3,1,-2,-1), \, (+,3,-1,-2,-5), \, (-,4,-4,2,1). $ \\[2mm]

\end{tabular}  \\  \arrayrulecolor{black!30}\midrule

$\mathbf{3}$ & \begin{tabular}[c]{@{}c@{}}

\\[-2mm]
 $(+,0,0,0,2), \, (-,0,0,-1,1), \, (+,0,0,2,0), \, (+,0,0,-2,2), \, (-,0,1,0,1), $ \\ 
 $(+,0,-1,2,0), \, (-,0,2,1,0), \, (+,1,0,-2,2), \, (+,1,-1,0,0), \, (+,2,0,0,-1), $ \\ 
 $(-,-2,0,0,-1), \, (+,2,1,-1,-5), \, (-,2,-2,0,1), \, (-,-2,-2,1,-1),  $ \\ 
 $(-,2,2,2,-1), \, (+,-3,-1,1,0). $ \\[2mm]

\end{tabular}  \\  \midrule

$\mathbf{4}$ & \begin{tabular}[c]{@{}c@{}}

\\[-2mm]
 $(+,0,-1,0,2), \, (+,0,3,1,0), \, (-,0,-3,1,0), \, (+,1,0,-1,0), \, (+,1,0,2,0), $ \\ 
 $(-,-2,-1,1,-1), \, (+,2,5,4,0), \, (+,-3,0,0,-1). $ \\[2mm]

\end{tabular}  \\  \midrule

$\mathbf{5}$ & \begin{tabular}[c]{@{}c@{}}

\\[-2mm]
 $(+,0,1,-1,0), \, (-,0,-1,2,1), \, (-,0,-2,1,-1), \, (-,-1,0,-2,2), \, (-,-1,-1,4,1), $ \\ 
 $(+,1,-2,-1,1), \, (+,-2,0,-1,0), \, (-,3,-1,-2,1), \, (-,4,0,1,-1). $ \\[2mm]

\end{tabular}  \\  \midrule

$\mathbf{6}$ & \begin{tabular}[c]{@{}c@{}}

\\[-2mm]
 $(+,0,2,0,0), \, (-,3,0,0,-1). $ \\[2mm]

\end{tabular}  \\  %
\midrule

$\mathbf{7}$ & \begin{tabular}[c]{@{}c@{}}

\\[-2mm]
 $(-,-1,-2,3,0), \, (+,1,-3,-1,1), \, (+,-2,-2,1,0), \, (+,3,1,1,-1). $ \\[2mm]

\end{tabular}  \\  \midrule

$\mathbf{10}$ & \begin{tabular}[c]{@{}c@{}}

\\[-2mm]
 $(-,3,-2,1,-2). $ \\[2mm]

\end{tabular}  \\  \midrule

$\mathbf{11}$ & \begin{tabular}[c]{@{}c@{}}

\\[-2mm]
 $(+,8,-1,-1,-7). $ \\[2mm]

\end{tabular}  \\  \midrule

$\mathbf{15}$ & \begin{tabular}[c]{@{}c@{}}

\\[-2mm]
 $(+,2,3,4,2). $  \\[2mm]

\end{tabular}  \\ \midrule

$\mathbf{251}$ & \begin{tabular}[c]{@{}c@{}}

\\[-2mm]
 $(+,-4,10,0,-2). $ \\[2mm]

\end{tabular}  \\

\arrayrulecolor{black}\bottomrule

\end{longtable}

\end{document}